\documentclass[a4paper]{amsart} 
     
\usepackage{amssymb} 
\usepackage{newtxtext, newtxmath}
\usepackage{color}  
\usepackage[normalem]{ulem}
\usepackage[driverfallback=dvipdfm]{hyperref}

\usepackage{enumitem}             

\theoremstyle{plain}
\newtheorem*{thm}{Main Theorem}
\newtheorem{lem}{Lemma}[section]

\theoremstyle{definition} 
\newtheorem{defn}{Definition}

\theoremstyle{remark}

\newtheorem*{nota}{Notation}

\numberwithin{equation}{section}

\newcommand{\N}{\mathbb{N}}

\newcommand{\Czero}[1]{C_0(#1)}
\newcommand{\posC}[1]{C_0^+(#1)} 
\newcommand{\F}{\mathcal{F}} 

\DeclareMathOperator{\supp}{supp}
\DeclareMathOperator{\coz}{coz}

 \usepackage{marginnote}
 \usepackage{geometry}

\begin{document}

\title[Norm additive mappings]
{Norm additive mappings between the positive cones of continuous function algebras}

\author[N.~Shibata]{Natsumi Shibata}
\address{Graduate School of Science and Technology, Niigata University, Niigata 950-2181, Japan}
\email{f25a056h@mail.cc.niigata-u.ac.jp}
\author[T.~Miura]{Takeshi Miura}
\address{Department of Mathematics,
Faculty of Science, Niigata University,
Niigata 950-2181, Japan}
\email{miura@math.sc.niigata-u.ac.jp}


\begin{abstract}
We study bijections between the positive cones of spaces of continuous functions vanishing at infinity that satisfy a norm additive condition. 
Such maps arise naturally in the study of nonlinear functional equations and norm-preserving structures on function spaces.

While in the compact (unital) case these maps can often be analyzed via linear extension techniques, the non-unital setting $\Czero{X}$ requires a different approach due to the absence of a distinguished unit element. 

In this paper, we show that every bijection $T:\posC{X}\to\posC{Y}$
between the positive cones of $\Czero{X}$ and $\Czero{Y}$
satisfying 
\[
\|T(f+g)\|=\|Tf+Tg\|
\]
for all $f,g\in\posC{X}$ admits a representation of the form
\[
Tf(y)=h(y)f(\tau(y)),
\]
where $\tau:Y\to X$ is a homeomorphism and
$h$ is a bounded continuous function from $Y$ to $(0,\infty)$.

This yields a complete characterization of norm additive bijections
on positive cones of $\Czero{X}$.
\end{abstract}

\maketitle

\section{Introduction}

Let $X$ be a locally compact Hausdorff space, and let $\Czero{X}$ denote the Banach space of continuous real-valued functions vanishing at infinity, equipped with the supremum norm
$\|f\|=\sup_{x\in X}|f(x)|$ for $f\in\Czero{X}$.
We denote by $\posC{X}$ the positive cone of $\Czero{X}$.
That is,
\[
\posC{X}=\{f\in\Czero{X}:f(x)\geq0\quad\mbox{for all $x\in X$}\}.
\]
In this paper, we study nonlinear bijections on $\posC{X}$ preserving a norm additive structure.

Mappings between function spaces preserving norm-type structures have been studied extensively in functional analysis. 
In particular, nonlinear transformations satisfying additive-type conditions on positive cones arise in the study of isometric structures and nonlinear variants of classical functional equations
(see e.g., \cite{Aczel,FischerMuszely}).

A typical example is the norm additive condition
\[
\|T(f+g)\|=\|Tf+Tg\|
\qquad
(f,g\in\posC{X}).
\]
Such mappings have been investigated in several contexts. 
Among related works, Moln\'{a}r \cite{Molnar} investigated nonlinear maps preserving various norm-type structures, while Hirota \cite{Hirota} studied norm additive
surjections on positive cones of function spaces and established their additivity and positive homogeneity.

In the case where $X$ is compact, such maps can often be described using linear extension techniques. 
However, the situation becomes more delicate for the non-unital space $\Czero{X}$. 
Since $\Czero{X}$ does not contain a constant
unit function, these techniques are no longer directly applicable, 
and one cannot reduce the problem to the study of linear operators in a straightforward manner.

The aim of this paper is to give a complete description of bijections
\[
T\colon\posC{X}\to\posC{Y}
\]
satisfying the norm additive condition in this non-unital setting.

We now state the main result of this paper, which gives a complete description of such norm additive bijections.

\begin{thm}
Let $X$ and $Y$ be locally compact Hausdorff spaces,
and let $T \colon \posC{X} \to \posC{Y}$ be a bijection satisfying
\[
    \|T(f+g)\| = \|Tf + Tg\|
\]
for all $f, g \in \posC{X}$.
Then there exist a homeomorphism $\tau \colon Y \to X$ and a bounded continuous function $h \colon Y \to (0, \infty)$ bounded away from zero such that
\[
    Tf(y) = h(y)f(\tau(y))
\]
for all $f \in \posC{X}$ and $y \in Y$.

Conversely, any mapping $T$ of this form is a bijection from $\posC{X}$ onto $\posC{Y}$ satisfying the norm additive condition as above.
\end{thm}

The proof is based on a detailed analysis of the algebraic and order structure induced by the norm condition. 
We first derive structural properties of $T$, then analyze its interaction with supports of functions, which leads to the construction of the underlying homeomorphism, 
and finally derive the weighted composition representation of $T$.

\section{Preliminaries and structural properties}

In this section, we derive several preliminary and structural properties of norm additive bijections on $\posC{X}$. 
Our purpose is to establish the basic algebraic and support behavior of such maps, which will be essential for constructing the underlying homeomorphism in the subsequent sections.

We first fix some notation concerning supports of functions.

\begin{nota}
For a function $f\in\posC{X}$, we write
\[
    \coz(f)=\{x\in X:f(x)\ne 0\}, \qquad
    \supp(f)=\overline{\coz(f)},
\]
where $\overline{\cdot}$ denotes the closure.
Thus $x\notin \supp(f)$ if and only if $f$ vanishes on some open neighborhood of $x$.
\end{nota}

Throughout this section, $X$ and $Y$ are locally compact Hausdorff spaces, and
\[
T\colon \posC{X}\to \posC{Y}
\]
is a bijection satisfying
\[
\|T(f+g)\|=\|Tf+Tg\|
\qquad (f,g\in\posC{X}).
\]

We begin by recalling Hirota's theorem, which shows that the norm additive condition already forces $T$ to be additive and positively homogeneous on $\posC{X}$.

\begin{lem}[{Hirota \cite[Theorem 1.1]{Hirota}}]\label{lem:hirota_reduction}
Let $T \colon \posC{X} \to \posC{Y}$ be a surjective mapping satisfying
\[
    \|T(f+g)\| = \|Tf + Tg\|
\]
for all $f, g \in \posC{X}$. Then $T$ is additive and positively homogeneous;
in particular, 
\[
T(rf)=rTf
\]
for all $r\geq0$ and $f\in\posC{X}$.
\end{lem}

Hence, throughout the remainder of the paper, both $T$ and $T^{-1}$ may be regarded as additive and positively homogeneous.

We first observe that this algebraic rigidity immediately implies preservation of the natural order structure.

\begin{lem}\label{lem:order_iso}
The mapping $T$ is an order isomorphism; namely, for all $f,g\in\posC{X}$,
\[
f\le g \quad \text{if and only if} \quad Tf\le Tg.
\]
\end{lem}

\begin{proof}
Assume that $f\le g$. Then $g-f\in\posC{X}$, and hence, by additivity of $T$,
\[
Tg=T(f+(g-f))=Tf+T(g-f)\ge Tf,
\]
since $T(g-f)\in\posC{Y}$. Thus $Tf\le Tg$.

The converse follows by applying the same argument to $T^{-1}$.
\end{proof}

Having established order preservation, we next investigate preservation of finite disjointness, where the positive cone starts to reflect the topology of the underlying spaces.
For this purpose, we first record a simple characterization of finite disjointness for positive functions.
See \cite{FontHernandez,Jarosz} for disjointness preserving maps.

\begin{lem}\label{lem:ortho_char}
For $n\in\N$ and
$f_1, \dots, f_n \in \posC{X}$,
\[
    f_1 \cdots f_n = 0 \quad \text{if and only if} \quad \min\{f_1, \dots, f_n\} = 0.
\]
\end{lem}

\begin{proof}
Since all functions involved are nonnegative, the pointwise identity
\[
f_1(x)\cdots f_n(x)=0
\quad\text{if and only if}\quad
\min\{f_1(x),\dots,f_n(x)\}=0
\]
holds for every $x\in X$. Hence the conclusion follows.
\end{proof}

The next lemma shows that $T$ preserves finite disjointness; in particular, $T$ is biseparating.

\begin{lem}\label{lem:biseparating}
For $f_1, \dots, f_n \in \posC{X}$,
\[
    f_1 \cdots f_n = 0 \quad \text{if and only if} \quad Tf_1 \cdots Tf_n = 0.
\]
\end{lem}

\begin{proof}
Since $T$ is additive, we have $T0=0$, and similarly $T^{-1}0=0$.

Assume that $f_1\cdots f_n=0$. By Lemma~\ref{lem:ortho_char},
\[
\min\{f_1,\dots,f_n\}=0.
\]
Set
\[
v=\min\{Tf_1,\dots,Tf_n\}\in\posC{Y}.
\]
Then $v\le Tf_i$ for each $i$, and hence, by order preservation of $T^{-1}$,
\[
T^{-1}v\le f_i
\qquad (i=1,\dots,n).
\]
Therefore,
\[
0\le T^{-1}v\le \min\{f_1,\dots,f_n\}=0,
\]
which implies that $T^{-1}v=0$. Since $T^{-1}0=0$ and $T^{-1}$ is injective, we obtain $v=0$. Applying Lemma~\ref{lem:ortho_char} again, it follows that
\[
Tf_1\cdots Tf_n=0.
\]

The converse is obtained by the same argument applied to $T^{-1}$.
\end{proof}

For each $y\in Y$, consider
\[
\mathcal{F}_y=\{f\in\posC{X}\mid Tf(y)>0\}.
\]
Our goal is to show that
\[
\bigcap_{f\in\mathcal{F}_y}\supp(f)
\]
consists of exactly one point. 

To carry out this construction, we first need a compactly supported member of $\mathcal{F}_y$, and for this purpose we establish a boundedness property of $T$.

\begin{defn}\label{def:T_bounded}
An additive and positively homogeneous mapping 
\[
T\colon\posC{X}\to\posC{Y}
\]
is said to be bounded if there exists $M>0$ such that
\[
\|Tf\|\le M\|f\|
\]
for all $f\in\posC{X}$.
\end{defn}

The following elementary observation will be useful for verifying boundedness.

\begin{lem}\label{lem:T_bounded_equiv}
Let $T\colon\posC{X}\to\posC{Y}$ be additive and positively homogeneous. Then the following are equivalent:
\begin{enumerate}
    \item[(i)] $T$ is bounded;
    \item[(ii)] there exists a constant $M > 0$ such that
    \[
        \|Tu\| \le M
    \]
    for all $u \in \posC{X}$ with $\|u\| \le 1$.
\end{enumerate}
\end{lem}

\begin{proof}
(i) $\Rightarrow$ (ii) is immediate.

(ii) $\Rightarrow$ (i). Let $f\in\posC{X}$. If $f=0$, there is nothing to prove. Assume $f\neq0$, and put
\[
u=\frac{f}{\|f\|}.
\]
Then $u\in\posC{X}$ and $\|u\|=1$, so that $\|T(u)\|\le M$ by (ii). By positive homogeneity,
\[
\|T(f)\|=\|f\|\,\|T(u)\|\le M\|f\|.
\]
Hence $T$ is bounded.
\end{proof}

To obtain a compactly supported function in $\mathcal F_y$, we need quantitative control on the size of images under $T$. 
This is provided by the following boundedness result, whose proof relies on the completeness of $\Czero{X}$.

\begin{lem}\label{lem:T_bounded}
The mappings
\[
T\colon\posC{X}\to\posC{Y}
\quad\text{and}\quad
T^{-1}\colon\posC{Y}\to\posC{X}
\]
are both bounded.
\end{lem}

\begin{proof}
Assume to the contrary that $T$ is not bounded. Then, by Lemma~\ref{lem:T_bounded_equiv}, for each $n\in\mathbb N$ there exists $u_n\in\posC{X}$ such that $\|u_n\|\le1$ and
\[
\|Tu_n\|>n2^n.
\]

Set $f_n=2^{-n}u_n$. Then $\|f_n\|\le2^{-n}$, and hence the series $\sum_{n=1}^\infty f_n$ converges in norm to some $f\in\Czero{X}$, since $\Czero{X}$ is complete. Clearly $f\in\posC{X}$.

Moreover, $f\ge f_n$ for every $n$, so that by order preservation and positive homogeneity of $T$,
\[
T(f)\ge Tf_n=2^{-n}Tu_n\ge0.
\]
Therefore,
\[
\|Tf\|\ge 2^{-n}\|Tu_n\|>n
\]
for all $n\in\mathbb N$, which is impossible because $Tf\in\Czero{Y}$.

Thus $T$ is bounded. The boundedness of $T^{-1}$ follows by the same argument.
\end{proof}

As an immediate consequence of boundedness and order preservation, we obtain a Lipschitz-type continuity estimate for $T$.

\begin{lem}\label{lem:T_continuous}
There exists $M>0$ such that
\[
\|Tf-Tg\|\leq M\|f-g\|
\qquad(f,g\in\posC{X}).
\]
\end{lem}

\begin{proof}
Let $f,g\in\posC{X}$ and set $h=|f-g|\in\posC{X}$. Then
\[
f\le g+h
\qquad\text{and}\qquad
g\le f+h.
\]
By additivity and order preservation of $T$,
\[
Tf\le Tg+Th
\qquad\text{and}\qquad
Tg\le Tf+Th.
\]
Hence
\[
|Tf(y)-Tg(y)|\le Th(y)
\]
for all $y\in Y$, and therefore
\[
\|Tf-Tg\|\le \|Th\|.
\]
Since $T$ is bounded, there exists $M>0$ such that
\[
\|Th\|\le M\|h\|=M\|f-g\|.
\]
This proves the assertion.
\end{proof}

We now use the continuity estimate to produce a compactly supported function belonging to $\mathcal F_y$.

\begin{lem}\label{lem:compact_support_exist}
For each $y \in Y$, the family $\F_y$ contains a function with compact support.
\end{lem}

\begin{proof}
Fix $y\in Y$. Since $T$ is surjective, there exists $g\in\posC{X}\setminus\{0\}$ such that $Tg(y)>0$.

For each $n\in\mathbb N$, define
\[
f_n(x)=\max\{0,g(x)-1/n\}.
\]
Then $f_n\in\posC{X}$ and $\|g-f_n\|\le1/n$.

Since $g$ vanishes at infinity and is nonzero, the set
\[
K_n=\{x\in X\mid g(x)\ge1/n\}
\]
is a nonempty compact subset of $X$ for all sufficiently large $n$. Moreover,
\[
\supp(f_n)\subset K_n,
\]
so that $f_n$ has compact support for such $n$.

By Lemma~\ref{lem:T_continuous}, there exists $M>0$ such that
\[
\|Tg-Tf_n\|\le M\|g-f_n\|\le \frac{M}{n}.
\]
Hence
\[
Tf_n(y)\ge Tg(y)-\frac{M}{n}>0
\]
for all sufficiently large $n$. Therefore $f_n\in\mathcal F_y$.
\end{proof}

We are now ready to identify the unique point of $X$ encoded by the value of $Tf(y)$.

\begin{lem}\label{lem:support_intersection}
For each $y \in Y$, the intersection
\[
\bigcap_{f \in \F_y} \supp(f)
\]
consists of exactly one point.
\end{lem}

\begin{proof}
Fix $y\in Y$. We first show that the intersection is nonempty. Suppose to the contrary that
\[
\bigcap_{f\in\mathcal F_y}\supp(f)=\emptyset.
\]
By Lemma~\ref{lem:compact_support_exist}, choose $f_0\in\mathcal F_y$ with compact support. Then the family
\[
\{\supp(f)\cap\supp(f_0)\}_{f\in\mathcal F_y}
\]
consists of closed subsets of the compact set $\supp(f_0)$ with empty intersection. By compactness, there exist $f_1,\dots,f_m\in\mathcal F_y$ such that
\[
\supp(f_0)\cap\supp(f_1)\cap\cdots\cap\supp(f_m)=\emptyset.
\]
Hence
\[
f_0f_1\cdots f_m=0.
\]
Lemma~\ref{lem:biseparating} yields
\[
Tf_0Tf_1\cdots Tf_m=0.
\]
However, all factors are strictly positive at $y$, a contradiction.

It remains to prove uniqueness. Suppose that
\[
x_1,x_2\in\bigcap_{f\in\mathcal F_y}\supp(f)
\quad\text{with}\quad x_1\ne x_2.
\]
Choose disjoint open neighborhoods $U_1$ and $U_2$ of $x_1$ and $x_2$, respectively. By Urysohn's lemma, there exist $w_1,w_2\in\posC{X}$ such that
\[
0\le w_i\le1,\qquad w_i=1 \text{ on a neighborhood of }x_i,\qquad \supp(w_i)\subset U_i
\]
for $i=1,2$. In particular, $w_1w_2=0$.

Fix $f\in\mathcal F_y$ and set $g_i=fw_i$ $(i=1,2)$. Then $g_1g_2=0$, and hence
\[
Tg_1Tg_2=0
\]
by Lemma~\ref{lem:biseparating}. Thus at least one of $Tg_1(y)$ and $Tg_2(y)$ is zero; assume $Tg_1(y)=0$.

Set $g=f-g_1=f-fw_1$. Then $g\in\posC{X}$ and
\[
Tg(y)=Tf(y)-Tg_1(y)=Tf(y)>0,
\]
where we have used additivity of $T$.
Hence $g\in\mathcal F_y$. On the other hand, since $w_1=1$ on a neighborhood of $x_1$, the function $g$ vanishes on a neighborhood of $x_1$, and therefore
\[
x_1\notin\supp(g),
\]
contrary to the definition of the intersection. This proves uniqueness.
\end{proof}

\begin{defn}\label{def:tau}
For each $y\in Y$, let $\tau(y)$ denote the unique point satisfying
\[
\bigcap_{f\in\mathcal F_y}\supp(f)=\{\tau(y)\}.
\]
This defines a mapping
\[
\tau\colon Y\to X.
\]
\end{defn}

The defining property of $\tau$ immediately yields the following locality relation:
\begin{equation}\label{eq:locality}
    \tau(y)\notin \supp(f) \quad \Longrightarrow \quad Tf(y)=0
    \qquad (f\in \posC{X},\ y\in Y).
\end{equation}
Indeed, if $Tf(y)>0$, then $f\in\mathcal F_y$, and hence $\tau(y)\in\supp(f)$ by definition.

We next show that the point map $\tau$ obtained above is continuous.

\begin{lem}\label{lem:tau_continuous}
The mapping $\tau \colon Y \to X$ is continuous.
\end{lem}

\begin{proof}
Let $U\subset X$ be open and $y_0\in\tau^{-1}(U)$. Set $x_0=\tau(y_0)\in U$.
It follows from the definition of the mapping $\tau$ that
\[
    \bigcap_{f \in \F_{y_0}} \supp(f)=\{\tau(y_0)\}=\{x_0\}.
\]
By Lemma~\ref{lem:compact_support_exist}, choose $f_0\in\mathcal F_{y_0}$ with compact support.
Because $x_0\in U$, we obtain $x_0\not\in\supp(f_0)\setminus U$.
Since
\[
\bigcap_{f\in\mathcal F_{y_0}}\supp(f)=\{x_0\}
\]
and $x_0\notin \supp(f_0)\setminus U$, we obtain
\[
\Bigl(\bigcap_{f\in\mathcal F_{y_0}}\supp(f)\Bigr)\cap(\supp(f_0)\setminus U)=\emptyset.
\]
The family
$    \{ \supp(f) \cap (\supp(f_0) \setminus U):f \in \F_{y_0}\}$
consists of closed subsets of the compact set $\supp(f_0)$
and has empty intersection by the above equality.
Hence there exist
$f_1, \dots, f_n \in \F_{y_0}$ such that
\[
    \left( \bigcap_{i=1}^n \supp(f_i) \right) \cap (\supp(f_0) \setminus U) = \emptyset.
\]
This implies that
\[
    \bigcap_{i=0}^n \supp(f_i) \subset U.
\]

Set
\[
    W = \bigcap_{i=0}^n \{ y \in Y \mid Tf_i(y) > 0 \}.
\]
Then $W$ is an open neighborhood of $y_0$, since $f_i\in\F_{y_0}$ for each $i$.
If $y \in W$, then $f_i \in \F_y$ for all $i \in \{0, \dots, n\}$,
and therefore
\[
    \tau(y) \in \bigcap_{i=0}^n \supp(f_i) \subset U.
\]
Thus $W\subset\tau^{-1}(U)$, showing that $\tau^{-1}(U)$ is open. Hence $\tau$ is continuous.
\end{proof}

\begin{lem}\label{lem:homeo}
The mapping $\tau\colon Y\to X$ is a homeomorphism.
\end{lem}

\begin{proof}
Apply the preceding construction to
\[
T^{-1}\colon \posC{Y}\to\posC{X}.
\]
For each $x\in X$, set
\[
\mathcal G_x=\{v\in\posC{Y}:T^{-1}v(x)>0\}.
\]
Then there exists a continuous map $\sigma\colon X\to Y$ such that
\[
\bigcap_{v\in\mathcal G_x}\supp(v)=\{\sigma(x)\}.
\]
Moreover,
\begin{equation}\label{eq:locality_inv}
\sigma(x)\notin\supp(v)
\quad\Longrightarrow\quad
T^{-1}v(x)=0
\qquad (v\in\posC{Y},\ x\in X).
\end{equation}

We show that $\tau\circ\sigma=\mathrm{id}_X$. Suppose, to the contrary, that
\[
\tau(\sigma(x_0))\ne x_0
\]
for some $x_0\in X$, and put $x_1=\tau(\sigma(x_0))$.
Choose disjoint open neighborhoods $U_0$ and $U_1$ of $x_0$ and $x_1$, respectively. By Urysohn's lemma, there exists $u\in\posC{X}$ such that
\[
u(x_0)>0,\qquad \supp(u)\subset U_0.
\]
Then $\tau(\sigma(x_0))=x_1\notin\supp(u)$.

Since $\tau$ is continuous, 
\[
V=\tau^{-1}(X\setminus\supp(u))
\]
is an open neighborhood of $\sigma(x_0)$. For every $y\in V$, we have
$\tau(y)\notin\supp(u)$, and hence \eqref{eq:locality} gives $Tu=0$
on $V$. Thus
$V\cap\supp(Tu)=\emptyset$, in particular
 $\sigma(x_0)\notin\supp(Tu)$. Applying \eqref{eq:locality_inv} to
$v=Tu$, we obtain
\[
u(x_0)=T^{-1}(Tu)(x_0)=0,
\]
which contradicts $u(x_0)>0$. Hence $\tau\circ\sigma=\mathrm{id}_X$.

By symmetry, $\sigma\circ\tau=\mathrm{id}_Y$. Therefore $\tau$ is bijective and $\tau^{-1}=\sigma$ is continuous. Thus $\tau$ is a homeomorphism.
\end{proof}

\section{The weight and the representation formula}

Having constructed the homeomorphism $\tau\colon Y\to X$, we now determine the pointwise form of $T$. 
The remaining task is to identify the weight function associated with $T$.

\begin{lem}\label{lem:local_order_preserving}
Let $y\in Y$ and set $x_0=\tau(y)$. 
If $f,g\in\posC{X}$ satisfy $f\le g$ on an open neighborhood of $x_0$, then
\[
Tf(y)\le Tg(y).
\]
\end{lem}

\begin{proof}
Define
\[
k(x)=\max\{0,f(x)-g(x)\}.
\]
Then $k\in\posC{X}$ and $f\le g+k$. By order preservation and additivity,
\[
Tf(y)\le Tg(y)+Tk(y).
\]
Since $f\le g$ on a neighborhood of $x_0$, the function $k$ vanishes on a neighborhood of $x_0$. Hence $x_0\notin\supp(k)$, and \eqref{eq:locality} gives $Tk(y)=0$. Therefore
$Tf(y)\le Tg(y)$.
\end{proof}

For each $y\in Y$, local compactness of $X$ allows us to choose a function
$e\in\posC{X}$ which is equal to $1$ on a neighborhood of $\tau(y)$.
We next show that the value $Te(y)$ does not depend on this choice.

\begin{lem}\label{lem:weight_well_defined}
Let $y\in Y$ and set $x_0=\tau(y)$. 
If $e_1,e_2\in\posC{X}$ are equal to $1$ on neighborhoods of $x_0$, then
\[
Te_1(y)=Te_2(y).
\]
\end{lem}

\begin{proof}
Choose neighborhoods $U_i$ of $x_0$ such that $e_i=1$ on $U_i$ for $i=1,2$.
On $U_1\cap U_2$, we have both $e_1\le e_2$ and $e_2\le e_1$. 
Lemma~\ref{lem:local_order_preserving} therefore gives
\[
Te_1(y)\le Te_2(y)
\quad\text{and}\quad
Te_2(y)\le Te_1(y),
\]
and hence the desired equality follows.
\end{proof}

\begin{defn}\label{def:weight_function}
For each $y\in Y$, define
\[
h(y)=Te(y),
\]
where $e\in\posC{X}$ is any function equal to $1$ on a neighborhood of $\tau(y)$.
\end{defn}

This is well defined by Lemma~\ref{lem:weight_well_defined}.

\begin{lem}\label{lem:representation_formula}
For every $f\in\posC{X}$ and $y\in Y$, we have
\[
Tf(y)=h(y)f(\tau(y)).
\]
\end{lem}

\begin{proof}
Let $f\in\posC{X}$ and $y\in Y$, and put $x_0=\tau(y)$. 
Fix $\epsilon>0$. By continuity of $f$ at $x_0$, there exists an open neighborhood $V$ of $x_0$ such that
\[
f(x_0)-\epsilon\le f(x)\le f(x_0)+\epsilon
\qquad (x\in V).
\]
Choose $e\in\posC{X}$ such that $e=1$ on
an open neighborhood $W$ of $x_0$ with $W\subset V$. Since $f\ge0$, on $W$ we have
\[
\max\{0,f(x_0)-\epsilon\}e
\le f
\le (f(x_0)+\epsilon)e.
\]
By Lemma~\ref{lem:local_order_preserving} and positive homogeneity,
\[
\max\{0,f(x_0)-\epsilon\}Te(y)
\le Tf(y)
\le (f(x_0)+\epsilon)Te(y).
\]
Letting $\epsilon\downarrow0$
with $Te(y)=h(y)$ yields
\[
Tf(y)=h(y)f(x_0)=h(y)f(\tau(y)).
\]
This proves the assertion.
\end{proof}

\begin{lem}\label{lem:weight_properties}
The function $h$ is continuous, strictly positive, bounded, and bounded away from zero.
\end{lem}

\begin{proof}
We first show that $h(y)>0$ for all $y\in Y$. 
Choose $v\in\posC{Y}$ with $v(y)>0$. By surjectivity of $T$, there exists $f\in\posC{X}$ such that $Tf=v$. By Lemma~\ref{lem:representation_formula},
\[
0<v(y)=Tf(y)=h(y)f(\tau(y)),
\]
and hence $h(y)>0$.

Next we prove continuity. Let $y_0\in Y$. Choose $f\in\posC{X}$ such that
$f(\tau(y_0))>0$. Since $f\circ\tau$ is continuous, there is a neighborhood
$W$ of $y_0$ such that $f(\tau(y))>0$ for all $y\in W$.
By Lemma~\ref{lem:representation_formula}, we have
\[
h(y)=\frac{Tf(y)}{f(\tau(y))}
\]
for all $y\in W$.
Thus $h$ is continuous at $y_0$.

By Lemma~\ref{lem:T_bounded}, there exists $M>0$ such that
\[
\|Tf\|\le M\|f\|
\qquad (f\in\posC{X}).
\]
For each $y\in Y$, choose $u\in\posC{X}$ such that $u(\tau(y))=1$ and $\|u\|\le1$. Then
\[
h(y)=h(y)u(\tau(y))=Tu(y)\le \|Tu\|\le M.
\]
Hence $h$ is bounded above.

Applying the preceding construction to $T^{-1}$, we obtain a continuous function
$w\colon X\to(0,\infty)$ and a constant $C>0$ such that
\[
T^{-1}v(x)=w(x)v(\tau^{-1}(x))
\qquad (v\in\posC{Y},\ x\in X)
\]
and $w(x)\le C$ for all $x\in X$. From $T(T^{-1}v)=v$, we get
\[
h(y)w(\tau(y))v(y)=
h(y)(T^{-1}v)(\tau(y))=
v(y)
\]
for all $v\in\posC{Y}$ and $y\in Y$. Choosing $v$ with $v(y)>0$, we obtain
\[
h(y)w(\tau(y))=1.
\]
Thus
\[
\frac1{h(y)}=w(\tau(y))\le C,
\]
so $h(y)\ge 1/C$ for all $y\in Y$.
\end{proof}

\begin{proof}[\textbf{Proof of Main Theorem}]
The necessity is immediate
from Lemmas~\ref{lem:homeo}, \ref{lem:representation_formula}, and \ref{lem:weight_properties}.

Conversely, suppose that
\[
Tf(y)=h(y)f(\tau(y))
\]
for a homeomorphism $\tau\colon Y\to X$ and a continuous function
$h\colon Y\to(0,\infty)$ which is bounded and bounded away from zero.

For $f\in\posC{X}$, the function $f\circ\tau$ belongs to $C_0(Y)$, and since $h$ is bounded and continuous, we have
\[
Tf=h\cdot(f\circ\tau)\in\posC{Y}.
\]
Moreover,
\[
T(f+g)(y)=h(y)(f+g)(\tau(y))
=Tf(y)+Tg(y),
\]
so in particular
\[
\|T(f+g)\|=\|Tf+Tg\|.
\]

Define $S\colon\posC{Y}\to\posC{X}$ by
\[
Sv(x)=\frac{1}{h(\tau^{-1}(x))}v(\tau^{-1}(x)).
\]
Since $h$ is bounded away from zero, $Sv\in\posC{X}$ for all $v\in\posC{Y}$.
A direct computation gives
\[
S(Tf)=f,\qquad T(Sv)=v.
\]
Therefore, $T$ is a bijection from $\posC{X}$ onto $\posC{Y}$ satisfying the norm additive condition.
\end{proof}

\end{document}